\documentclass{amsart}
\usepackage{amssymb,amsthm,amsfonts,verbatim,graphicx,mathtools,mathrsfs,commath,float,mathabx}
\usepackage{enumerate,setspace}
\usepackage[hidelinks]{hyperref}
\usepackage{amssymb,amsfonts,verbatim,graphicx,mathtools,mathrsfs,commath,float,tikz-cd,amsthm,setspace}
\usepackage{caption}
\usepackage{subcaption}
\usepackage{tikz}
\usepackage{enumitem}
\usepackage{wrapfig}
\usepackage[hidelinks]{hyperref}
\usepackage{xfrac}
\usepackage[margin=1in]{geometry}
\makeatletter
\g@addto@macro\th@definition{\thm@headpunct{\textnormal{.}}}
\makeatother

\makeatletter
\def\ps@pprintTitle{%
 \let\@oddhead\@empty
 \let\@evenhead\@empty
 \def\@oddfoot{\centerline{\thepage}}%
 \let\@evenfoot\@oddfoot}
\makeatother

\DeclareMathOperator{\cl}{cl}
\DeclareMathOperator{\ord}{ord}

\numberwithin{equation}{section}
\date{\today}
\makeatletter
\@namedef{subjclassname@2010}{%
  \textup{2010} Mathematics Subject Classification}
\makeatother

\newtheorem{ut}{Theorem}
\newtheorem{ul}{Lemma}

\newtheorem{uq}{Question}

\theoremstyle{remark}
\newtheorem*{ur}{Remark}

\numberwithin{equation}{subsection}

\begin{document}

\title{Embedding irreducible connected sets}
\author{David Sumner Lipham}
\address{Department of Mathematics and Statistics, Auburn University, Auburn,
AL 36849, United States of America}
\email{dsl0003@auburn.edu}

\maketitle

\renewcommand{\thefootnote}{}



\renewcommand{\thefootnote}{\arabic{footnote}}
\setcounter{footnote}{0}
\vspace{-17mm}
\begin{abstract}
We show that every connected set $X$ which is irreducible between two points $a$ and $b$ embeds into the Hilbert cube in a way that $X\cup \{c\}$ is irreducible between $a$ and $b$ for every point $c$ in the closure of $X$.  Also, a connected set $X$ is indecomposable if and only if for every compactum $Y\supseteq X$ and $a\in X$ there are two points $b,c\in \overline X$ such that $X\cup \{b,c\}$ is irreducible between every two points from $\{a,b,c\}$.   Following the proofs,  we illustrate a cube embedding of the main example from  \textit{On indecomposability of $\beta X$}.   We  prove the example embeds into the plane.
\end{abstract}

\section{Introduction}

 The first half of this paper concerns one and two-point  extensions of connected sets.  All spaces are assumed to be separable and metrizable. 

Given a connected set $X$ and points $p,q\in X$, then $X$ is \textit{reducible between $p$ and $q$} if there is a connected $C\subseteq X$ with $\{p,q\}\subseteq C $ and $\overline C\neq X$.  If there is no such $C$,  then $X$ is \textit{irreducible between $p$ and $q$}. 

For a connected set $X\cup \{p\}$ to be irreducible between two points  $x\in X$ and $p$, it is clearly necessary that $p$ miss the closure of every proper component of $X$ which contains $x$. By $C$\textit{ is a proper component of }$X$, we mean there is a proper closed $A\subsetneq X$ such that $C$ is a connected component of $A$. This condition, however,  is not sufficient. Example 1 of \cite{lip2} describes a connected set $X\subseteq \mathbb R ^3$ such that every proper component of $X$ is degenerate, yet  $X\cup \{\langle 0,0,0\rangle\}$ is reducible.\footnote{This $X$ was denoted $q[W]$ in \cite{lip2}. The mapping $q\restriction W$ was a homeomorphic embedding of the widely-connected plane set  $W=W[X_3]$ from \cite{lip}.}  

 We prove the following.

 \begin{ut} \sloppy{For every connected set $X$ and  $x\in X$, there is a homeomorphic embedding $e:X\hookrightarrow [0,1]^\omega$ such that $e[X]\cup \{p\}$ is irreducible between $e(x)$ and $p$  for every point $p\in {\overline{e[X]}}$ which is not in the closure of any proper component of $e[X]$ containing $e(x)$.}
 \end{ut}
 
 
 \begin{ut} If  $X$ is  irreducible between two points $a$ and $b$, there is a homeomorphic embedding $e:X\hookrightarrow [0,1]^\omega$  such that $e[X]\cup \{c\}$ is irreducible between $e(a)$ and $e(b)$ for every  $c\in \overline{e[X]}$.\footnote{It is also possible to obtain $\dim({\overline{e[X]}})=\dim(X)$  in Theorems 1 and 2.}  \end{ut}
 
\begin{wrapfigure}[14]{R}{0.45\textwidth}
\centering\vspace{-0mm}
  \includegraphics[scale=0.45]{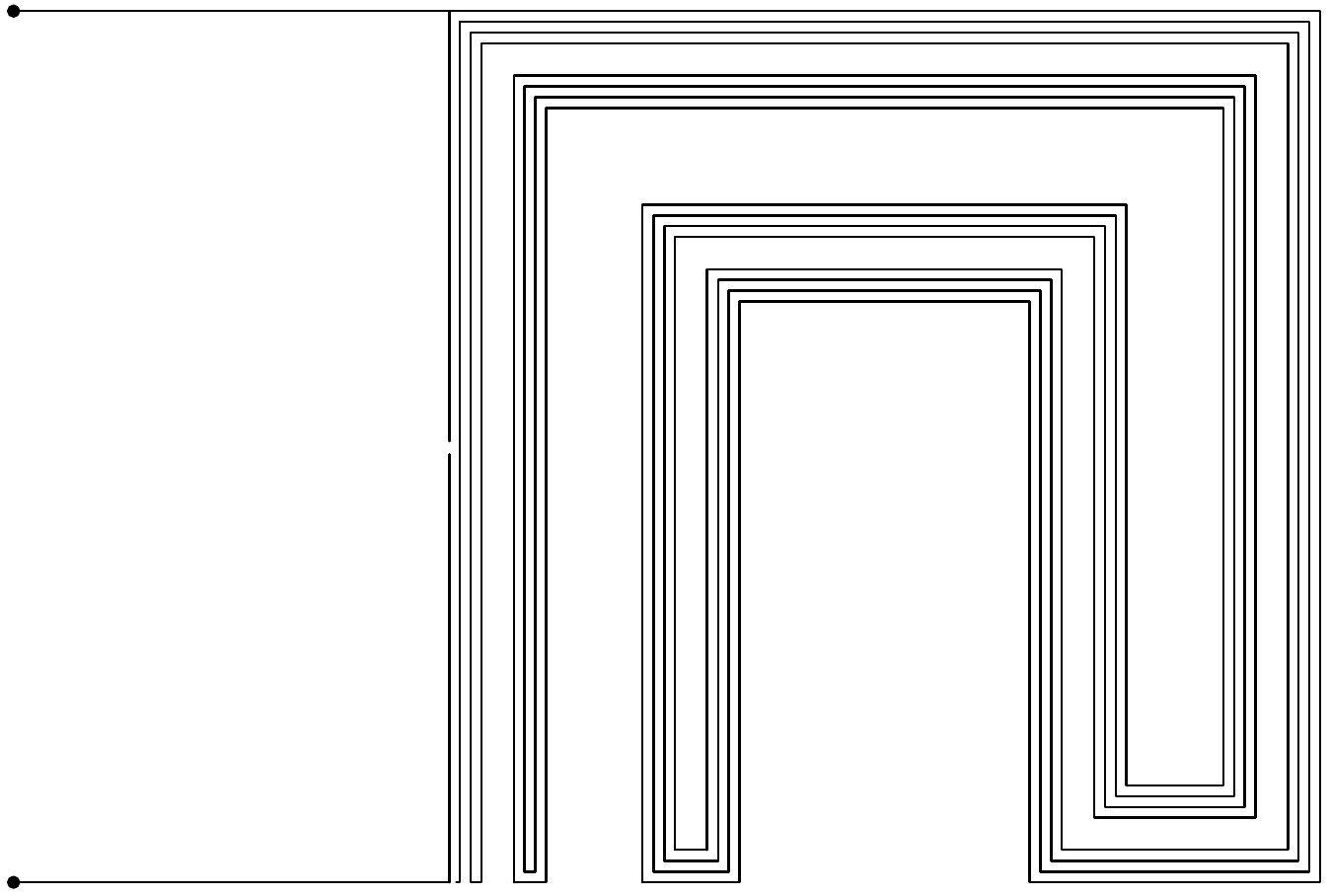} 
  \caption{$\mathfrak I\subseteq \mathbb R ^2$}
  \label{fiffe}
\end{wrapfigure}
The sharpness of Theorem 2 is demonstrated by the following example.  Let  $\mathfrak I\subseteq \mathbb R ^2$  be the rectilinear bucket-handle continuum spanning $[0,2]^2$, minus the point $\langle 0,1\rangle$, plus the two arcs $ [-1,0]\times \{0\}$ and $ [-1,0]\times \{2\}$. This connected set is locally compact and irreducible (between the two left-most points $\langle -1,0\rangle$ and $\langle -1,2\rangle$), yet its one-point compactification is reducible.\footnote{For an indecomposable example with similar properties, see the remark after   \cite[Example 1]{lip2}.}  This shows $e$ cannot be  arbitrary. Also,  the entire $\overline{e[\mathfrak I]}$ is reducible (between every two points of $\overline{e[\mathfrak I]}$) for every embedding $e$.  Thus,  adding too many limit points can destroy irreducibility regardless of the embedding. 

Note that every one-point extension of $\mathfrak I\setminus (\{0\}\times [\sfrac{1}{2},\sfrac{3}{2}])\simeq\mathfrak I$  is irreducible.

Following the proofs of Theorems 1 and 2, we generalize  an elementary result from continuum theory. A continuum $X$ is indecomposable if and only if there are three points $\{a,b,c\}\subseteq X$ such that $X$ is irreducible between every two points from $\{a,b,c\}$ (\cite[Corollary 11.20]{nad}, also \cite[\S48 VI Theorem 7$^\prime$]{kur}). 

\begin{ut}A connected separable metric space $X$ is indecomposable if and only if for every compactum $Y\supseteq X$ and $a\in X$ there are two points $b,c\in \overline X$ such that $X\cup \{b,c\}$ is irreducible between every two points from $\{a,b,c\}$.  \end{ut}

\noindent Here,  $X$ is \textit{indecomposable} means that  $X$ cannot be written as the union of two proper components. This is equivalent to saying $X$ is the only closed connected subset of $X$ with non-void interior  \cite[\S48 V Theorem 2]{kur}.  

In the second half of the paper we describe a connected set $\widehat W\subseteq [0,1]^3$ which has the same remarkable properties as $\widetilde W$ from \cite[Example 4]{lip2}.  Namely, $\widehat W$ is irreducible between every two of its points, but every continuum enclosing $\widehat W$ is reducible between every two points of $\widehat W$.    We indicate why $\widehat W$ has said properties in Section 4.2, after its construction in Section 4.1.  The construction technique here is similar to the one  in \cite{lip2} but is more geometric in nature. For all practical purposes $\widehat W$ is an embedding of $\widetilde W$, hence the title of Section 4: ``Embedding a special irreducible set''.  Planarity of $\widehat W$ is established in Section 4.3. 

We conclude with a brief discussion of $\leq 2$-point compactifications and some questions.

\section{Component ordinals}
In this section $X$ is a separable metric space. We introduce the notion of a \textit{component ordinal}, which will facilitate the proofs of Theorems 1 through 3.

For any point $x$ and space $X\ni x$,  let $C(x,X)$ denote the connected \textit{component} of $x$ in $X$.  That is, $C(x,X)=\bigcup \{C\subseteq X:C\text{ is connected and }x\in C\}.$ When $X$ is clear from context, we will simply write $C(x)$ for $C(x,X)$.

\begin{ul}For every $x\in X$ there is  a decreasing collection of closed sets $\{A^\alpha:\alpha<\omega_1\}$  such that  $A^0=X$, $(\bigcap_{\beta<\alpha}A^\beta)\setminus A^\alpha$ is closed for  each $\alpha<\omega_1$, and $\bigcap_{\alpha<\gamma}A^\alpha=C(x)$ for some $\gamma<\omega_1$.
\end{ul}

\begin{proof}Set $A_0=X$. 

Suppose $0<\alpha<\omega_1$ and $A^\beta $ has been defined for all $\beta <\alpha$.  If  $X^\alpha\coloneq \bigcap_{\beta<\alpha}A^\beta$ is not connected, let $A^\alpha \ni x$   be  a relatively clopen proper subset of $X^\alpha$. Otherwise, put $A^\alpha=X^\alpha$. 

The open $\big(X\setminus \bigcap_{\alpha<\omega_1}A^\alpha\big)$-cover $\{X\setminus A^\alpha:\alpha<\omega_1\}$ has a countable subcover, 
 meaning $(A^\alpha)_{\alpha<\omega_1}$  is eventually constant. 
 That is,   $X^\gamma$  is connected for some $\gamma<\omega_1$. Apparently $X^\gamma\subseteq C(x)$ by maximality of $C(x)$.  It remains to show $C(x)\subseteq X^\gamma$.  Well, if $C(x)\not \subseteq X^\gamma$ then there is a least $\beta< \gamma$ such that $C(x)\not \subseteq A^\beta$. Then $\beta>0$, $C(x)\subseteq X^\beta$, and the closed sets $A^\beta$ and $X^\beta\setminus A^\beta$ partition $C(x)$. This is absurd because $C(x)$ is connected. Therefore $C(x)\subseteq X^\gamma$ and the proof is complete.  \end{proof}

\begin{figure}[!htb]
    \centering
    \begin{minipage}{.3\textwidth}
    \includegraphics[scale=.4]{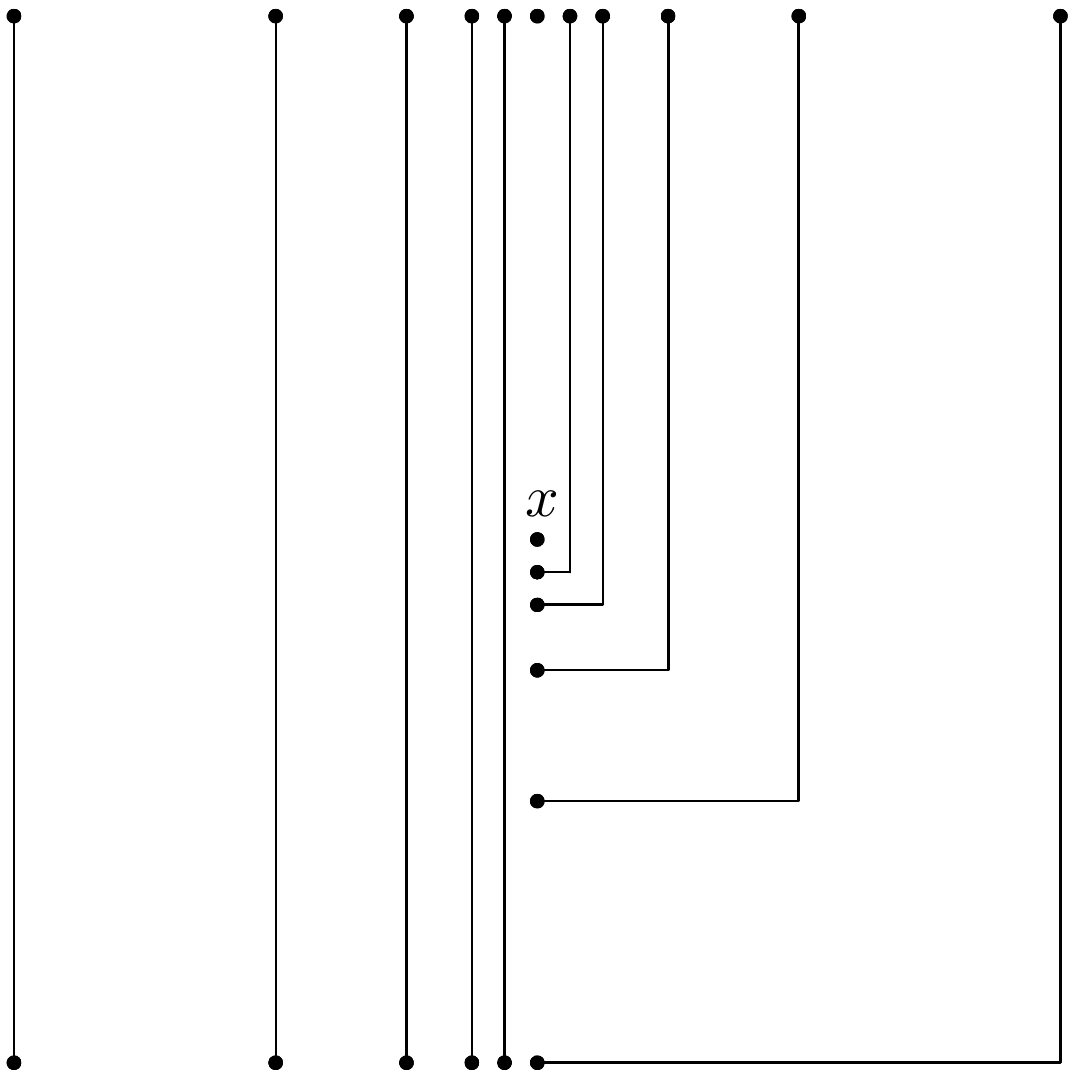}
         
              \
              
              \
              
                   \
                   
                        \
                        
                             \
                             
                                  \
                                  
                                       \
                                       
                                            \
                                            
                                                 \
                                                                                                                                                                                                                                                                                          
    \end{minipage}
    \begin{minipage}[b]{0.5\textwidth}
    \small
    
    \
    
      $\ord(x,\{x\})=\ord(x,\{x\}\cup (2\times[0,1]))=1$
      
      \

$\ord(x,[ \mathfrak X\cap ([0,\sfrac{1}{2})\times [0,1])]\cup \{x\})=\omega$

\

$\ord(x,[ \mathfrak X\cap ([0,\sfrac{1}{2})\times [0,1])]\cup\{x,\langle \sfrac{1}{2},1\rangle\})=\omega+1$

\

$\ord(x,[ \mathfrak X\cap ([0,\sfrac{1}{2}]\times [0,1])])=\omega+\omega$

\

$\ord(x, \mathfrak X\setminus \{ \langle \sfrac{1}{2},1\rangle\})=\omega+\omega$

\

$\ord(x, \mathfrak X)=\omega+\omega+1$
    \end{minipage}
    \vspace{-3.7cm}
     \caption{$\mathfrak X\subseteq [0,1]^2$ and the ordinal of $x$ in subspaces of $\mathfrak X$}
\end{figure}
Whenever $x$ is a point in a space $X$, we let $\ord(x,X)$ be the least (countable) ordinal $\gamma>0$ such that $X^\gamma=C(x)$, considering all $\omega_1$-sequences $(A^\alpha)$ from Lemma 1. 
Observe that in general,  $\ord$ is either  $1$,  a limit,  or the successor of a limit.   Also,  $C(x,X)$ is equal to the quasi-component of $x$ in $X$ if and only if $\ord(x,X)\in \{1,\omega\}$.  In Figure 2 we compute the order of   a particular point $x=\langle \sfrac{1}{2},\sfrac{1}{2}\rangle$ in a  few different subspaces of $[0,1]^2$.    In each space the component of $x$ is  $\{x\}$, but $\ord(x)$ varies.

\begin{ul}Let $\{A^\beta:\beta<\gamma\coloneq \ord(x,X)\}$ be given for a point $x\in X$ by Lemma 1.    If $X\cup \{p\}$ is a topological extension of $X$, and $p\in  C(x,X\cup \{p\})\setminus \overline {C(x)}$, then there exists $\beta<\gamma$ such that $p\in \overline {A^\beta}\cap \overline{X^\beta\setminus A^\beta}$. \end{ul}

\begin{proof}Suppose  $p\in  C(x,X\cup \{p\})\setminus \overline {C(x)}$.  

Put $M=C(x,X\cup \{p\})\setminus \{p\}$.  Then $p\in \overline M$, so $M\not\subseteq C(x)=X^\gamma$. As in the proof of Lemma 1, there is a least $\beta$ between $0$ and $\gamma$ such that $M\not\subseteq A^\beta$.   Then the $X$-closed sets $A^\beta$ and $X^\beta\setminus A^\beta$ partition $M$.  By connectedness of $C(x,X\cup \{p\})$ it must be that $p\in \overline {A^\beta}\cap \overline{X^\beta\setminus A^\beta}$.  \end{proof}

\section{Proofs of Theorems} 

Throughout, suppose $X$ is a connected separable metric space. 

Fix a basis $\{U_n:n<\omega\}$ for $X$ with each $U_n\neq\varnothing$. 

For each $n<\omega$ and $x\in X\setminus U_n$, let $\{A^\alpha_n(x):\alpha<\omega_1\}$ be given by Lemma 1,  with $A^0_n=X\setminus U_n$, $X^\alpha_n=\bigcap_{\beta<\alpha}A^\beta_n$,  and  $$X^{\ord(x,X\setminus U_n)}_n(x)=C_n(x)\coloneq C(x,X\setminus U_n).$$ 
If $x\in U_n$ then put $\ord(x,X\setminus U_n)=0$ and $C_n(x)=A^\alpha_n(x)=\varnothing$ for all $\alpha<\omega_1$. 

\subsection{Proof of Theorem 1}
Since $\gamma\coloneq \sup\{\ord(x,X\setminus U_n):n<\omega\}$ is countable, there exists  $e:X\hookrightarrow [0,1]^\omega$  such that ${\overline{e[A^\beta_n]}}\cap {\overline{e[X_n^\beta\setminus A^\beta_n]}}=\varnothing$ for every $n<\omega$ and $\beta<\gamma$.   If $p\in {\overline{e[X]}}\setminus \bigcup_{n<\omega} {\overline{e[C_n(x)]}}$, then  $p\notin \bigcup _{n<\omega}C(e(x),e[X\setminus U_n]\cup \{p\})$ by Lemma 2, so $e[X]\cup \{p\}$ is irreducible between $e(x)$ and $p$.  
\qed

\subsection{Proof of Theorem 2}

Let $a$ and $b$ be such that $X$ is irreducible between them. Put $$\xi=\sup\{\ord(x,X\setminus U_n):\langle x,n\rangle \in \{a,b\}\times\omega\},$$ and construct $e:X\hookrightarrow [0,1]^\omega$ so that 
\begin{align}&{\overline{e[A^\beta_n(x)]}}\cap {\overline{e[X^\beta_n(x)\setminus A^\beta_n(x)]}}=\varnothing\text{ for all }\langle x,n,\beta\rangle\in \{a,b\}\times \omega\times\xi\text{; and}\\
&\overline{e[C_n(a)]}\cap \overline{e[C_n(b)]}=\varnothing\text{ for all }n<\omega. 
\end{align}

Now let $c\in \overline{e[X]}$.  For any  $n<\omega$, we have 
$$C(e(a),e[X\setminus U_n]\cup \{c\})\cap C(e(b),e[X\setminus U_n]\cup \{c\})=\varnothing$$ by Lemma 2. Thus $e[X]\cup \{c\}$ is irreducible between $e(a)$ and $e(b)$. \qed
\medskip

\begin{ur}In Theorems 1 and 2 we can obtain  $\dim(\overline{e[X]})=\dim (X)$.  For if $\Phi$ is a countable family of continuous maps of $X$ into $[0,1]$, then $X$ has a $\Phi$-compactification $\overline{e[X]}$ (i.e. each member of  $\Phi$ continuously extends to $\overline{e[X]}$) with $\dim(\overline{e[X]})=\dim (X)$. See   \cite[Exercise 1.7.C]{eng3}.  

To  get ${\overline{e[A^\beta_n]}}\cap {\overline{e[X_n^\beta\setminus A^\beta_n]}}=\varnothing$ in the proof of Theorem 1, simply include a map $\varphi\in\Phi$ such that $\varphi[A^\beta_n]=0$ and $\varphi[X_n^\beta\setminus A^\beta_n]=1$. Likewise, $\Phi$ can guarantee (3.2.1) and (3.2.2)  in the proof of Theorem 2. \end{ur}

\subsection{Proof of Theorem 3} The condition is sufficient. If $X$ decomposes into two proper components  $H$ and $K$, then two of any three points $\{a,b,c\}\subseteq \overline{X}$ must belong to $\overline {H}$ or $\overline {K}$, and thus $X\cup \{a,b,c\}$  fails to have the three-point irreducible property. 

The condition is necessary.  Suppose $X$ is indecomposable.  

\textit{Case 1:} $X$ is reducible.   Fix $a\in X$.  Let $\gamma=\sup\{\ord(a,X\setminus U_n):n<\omega\}$.  We may select
$$b,c\in \overline X\setminus \bigcup \big\{\big[\overline{A^\beta_n(a)}\cap \overline{X^\beta_n(a)\setminus A^\beta_n(a)}\big]\cup \overline{C_n(a)}:\langle n,\beta\rangle\in\omega\times\gamma \big\},$$ since the sets $\overline{A^\beta_n(a)}\cap \overline{X^\beta_n(a)\setminus A^\beta_n(a)}$    and $\overline{C_n(a)}$ are closed \& nowhere dense in $\overline X$.  Then  by Lemma 2, $X\cup \{b,c\}$ is irreducible between $a$ and $b$, and $a$ and $c$.  It follows that $X\cup \{b,c\}$ is irreducible between $b$ and $c$.  For suppose $D$ is a proper closed connected subset of $X\cup \{b,c\}$ containing $b$ and $c$.  Since $X$ is reducible and  $D\cap X\neq\varnothing$, there is a closed connected $C\subsetneq X$ with $a\in C$ and $C\cap D\neq\varnothing$.  By indecomposability of $X$, $C$ is nowhere dense in $X$, whence $C\cup D\neq X$. Thus, the connected set $C\cup D$ violates the previously established irreducibility.

\textit{Case 2:} $X$ is irreducible.  Let $a$ and $b$ be such that $X$ is irreducible between them.  Let $\xi=\sup\{\ord(x,X\setminus U_n):\langle x,n\rangle\in\{a,b\}\times\omega\}$. Choose
$$c\in \overline X\setminus \bigcup \big\{\big[\overline{A_n^\beta(x)}\cap \overline{X_n^\beta(x)\setminus A_n^\beta(x)}\big]\cup \overline{C_n(x)}:
\langle x,n,\beta\rangle\in \{a,b\}\times \omega\times \xi\}.$$ Then $X\cup \{c\}$ is irreducible between $a$ and $c$, and $b$ and $c$. It follows easily that $X\cup \{c\}$ is also irreducible between $a$ and $b$. 
\qed

\section{Embedding a special irreducible set}

Here we will construct the set $\widehat W$ that  was outlined in Section 1.

 The set $\widehat W$ will be connected and irreducible between every two of its points.  The term for these properties is  \textit{widely-connected}.   One may see that a  connected set $W$  is widely-connected if and only if every subset of $W$ is either connected or hereditarily disconnected (a set is \textit{hereditarily disconnected} provided  each of its connected subsets is degenerate). 
 
 Bernstein subsets of certain continua such as the dyadic solenoid and bucket-handle are widely-connected.\footnote{For a Bernstein subset $B$ of a continuum $Y$ to be connected, every separator in $Y$ should be uncountable. For $B$ to be \textit{widely}-connected, it helps to know that every non-degenerate non-dense connected subset of $Y$  contains a non-degenerate continuum. That way, the only proper components of $B$ are degenerate.  Solenoids and bucket-handle continua satisfy all of these criteria since they have uncountably many disjoint composants, and all of their proper subcontinua are arcs.} Other examples, given in \cite{swi} and  \cite{lip},  are still less complicated than the one we are about to describe. But, like the Bernstein sets, they  fail to have the desired additional property 4.2.iii (item iii in  Section 4.2). $\widetilde W$  was the first example  with that property.

\subsection{Construction of $\widehat W$}

Let $S\subseteq[0,1]^2$ be the quinary (middle two-fifths) double bucket-handle continuum  depicted in Figure 3(a).\footnote{Parts of $S$ will play the roles of both $K$ and the orbit in $2^\mathbb Z$ from the  construction in \cite{lip2}.}  For other illustrations of $S$, see \cite{k} and \cite[\S48 V]{kur}.

\begin{figure}[h]
    \centering
    \begin{subfigure}[t]{0.3\textwidth}
        \centering
         \includegraphics[scale=0.292]{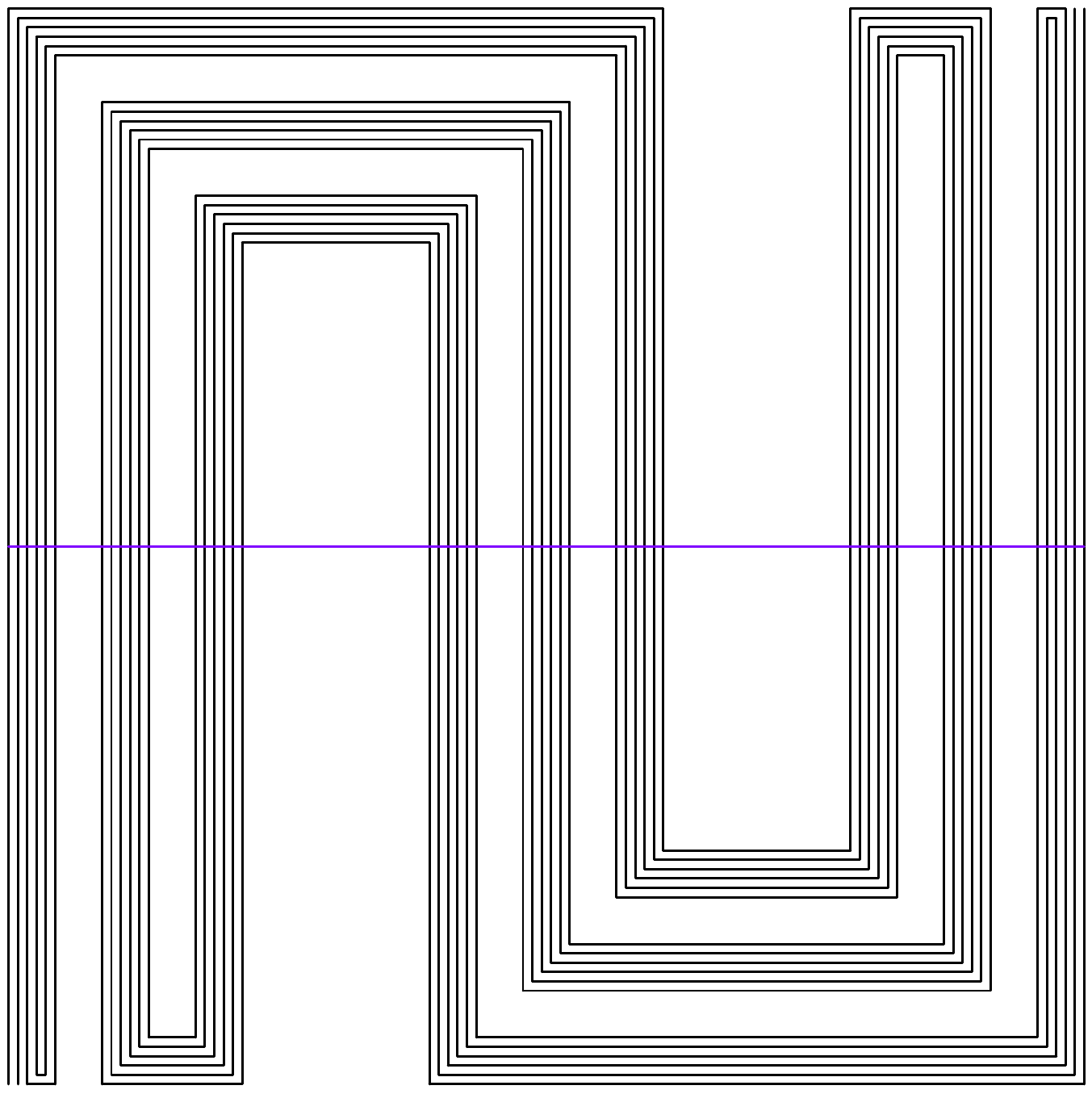}
        \caption{$S\cup (\color{violet}[0,1]\times \{\sfrac{1}{2}\}\color{black})$}
    \end{subfigure}%
    ~\hspace{.4mm}
    ~
    \begin{subfigure}[t]{0.3\textwidth}
        \centering
     \includegraphics[scale=0.3]{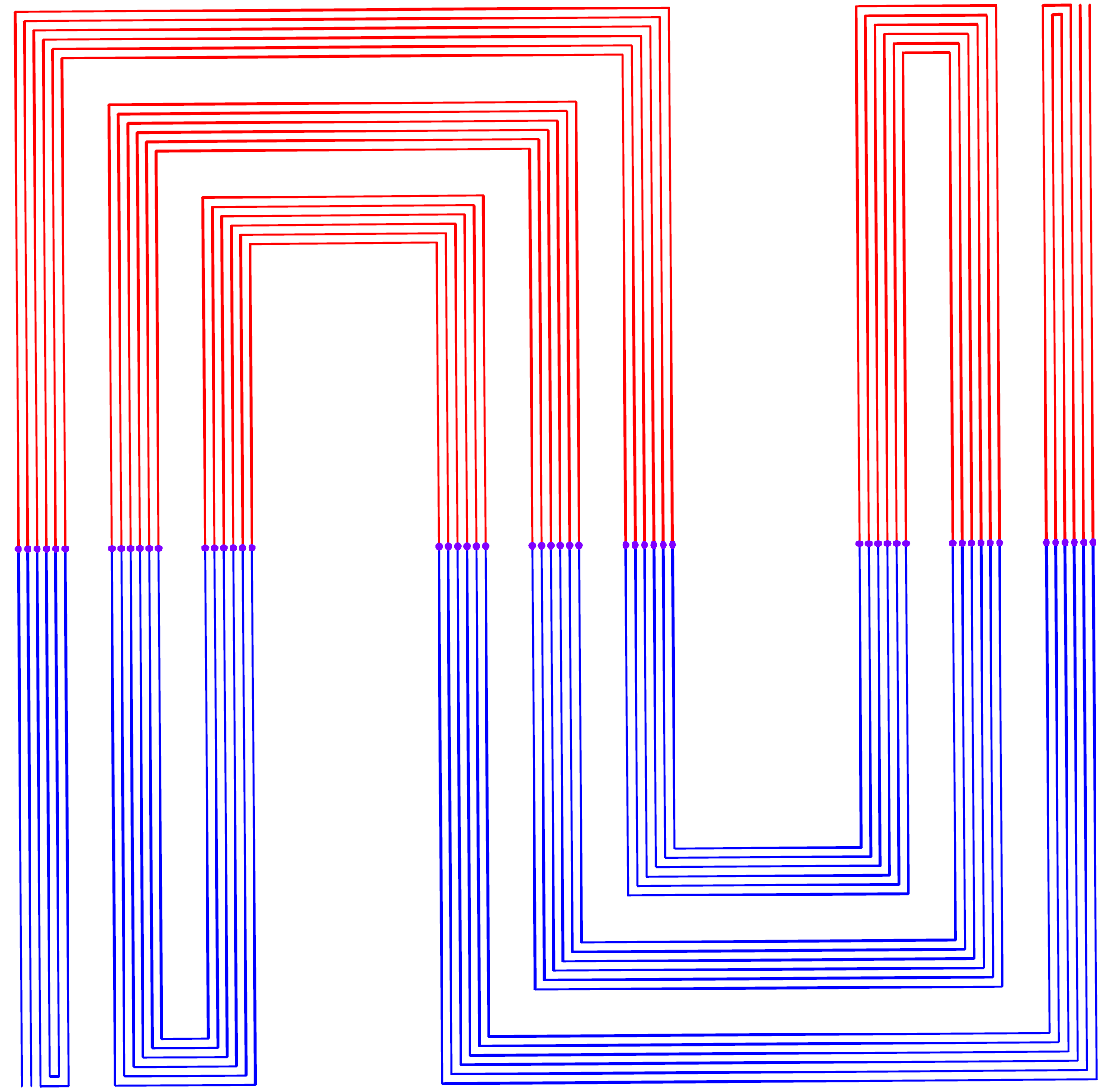}
        \caption{$\color{blue}W_0\color{black}\cup\color{red} W_1$}
    \end{subfigure}
    ~\hspace{.7mm}
    ~
       \begin{subfigure}[t]{0.3\textwidth}
        \centering
       \includegraphics[scale=0.285]{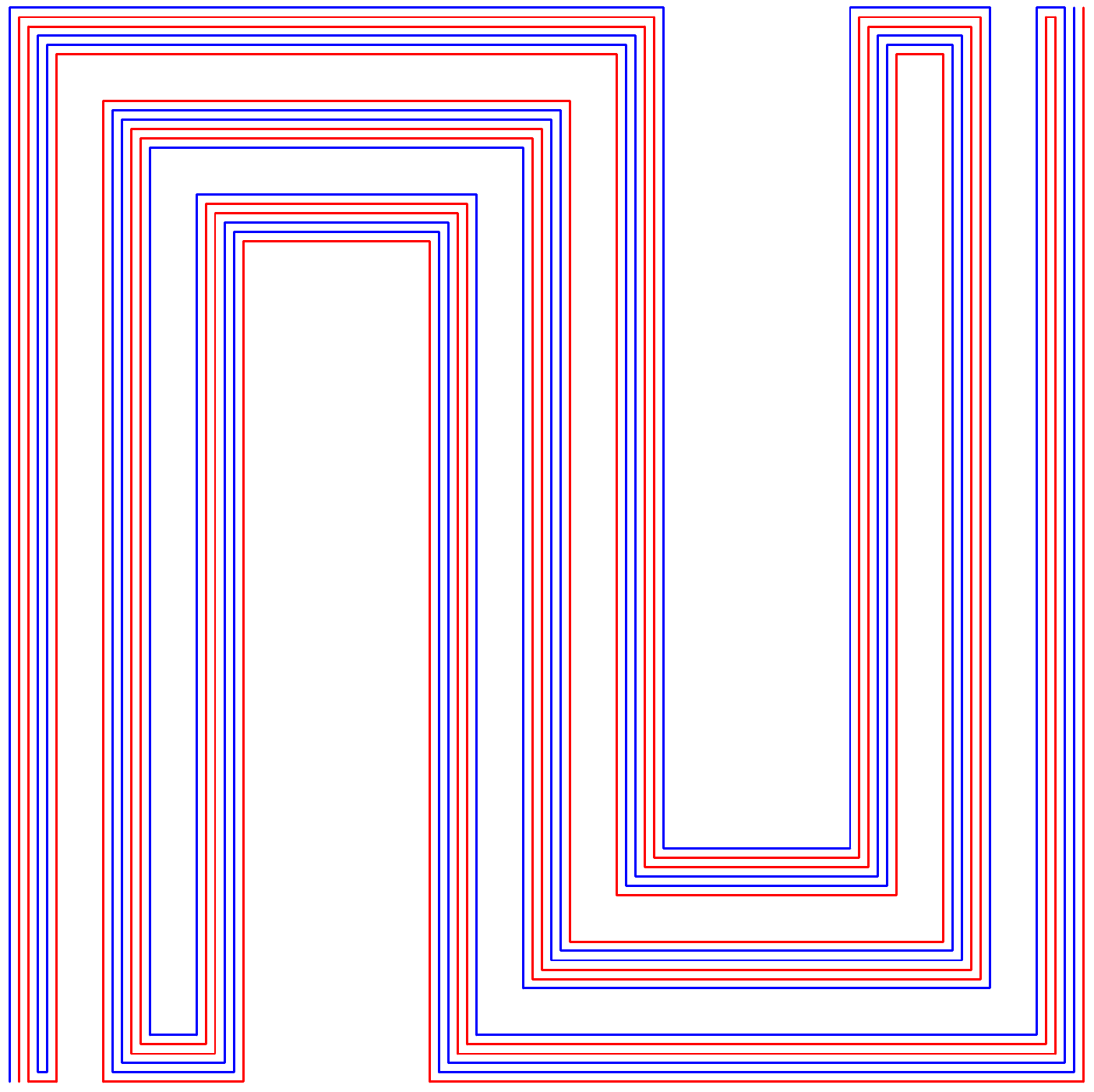}
          \caption{$\color{blue}P_0\color{black}\cup \color{red}P_1$}
    \end{subfigure}
    \caption{The double bucket-handle continuum and subsets}
\end{figure}

Let $W\supseteq S\cap( \color{violet}[0,1]\times  \{\sfrac{1}{2}\}\color{black})$ be a widely-connected subset of $S$.  

For visualization purposes,  construct $W$ as in  \cite{lip}: $S\setminus ( \color{violet}[0,1]\times  \{\sfrac{1}{2}\}\color{black})$ is the union of a countable discrete collection of sets homeomorphic to $C\times \mathbb R$,  $C$ being the middle two-fifths Cantor set. Refine each of these sets to a copy of the punctured Knaster-Kuratowski fan.  So if $c$ is an endpoint of an interval removed during the construction of $C$, then  the part of each $\{c\}\times \mathbb R$ remaining is $\{c\}\times \mathbb Q$; the other fibers become $\{c\}\times (\mathbb R\setminus \mathbb Q)$. To complete the construction of $W$, restore $S\cap( \color{violet}[0,1]\times  \{\sfrac{1}{2}\}\color{black})$.\footnote{The horizontal middle two-fifths Cantor set $S\cap (\color{violet}[0,1]\times  \{\sfrac{1}{2}\}\color{black})$ replaces the diagonal middle-thirds Cantor sets $\Delta$ and  $\Delta '$ from \cite{lip} and \cite{lip2}, respectively.}   

Let $\color{blue}W_0\color{black}=W\cap [0,\sfrac{1}{2}]$ and $\color{red}W_1=\color{black}W\cap [\sfrac{1}{2},1]$ be the lower and upper `halves' of $W$ depicted in Figure 3(b). So $W=\color{blue}{W_0}\color{black}\cup \color{red}W_1$ and $\color{blue}{W_0}\color{black}\cap\color{red}W_1\color{black}=S\cap (\color{violet}[0,1]\times \{\sfrac{1}{2}\}\color{black})$. 

The indecomposable  $S$ has two accessible  composants (maximal arcwise-connected subsets) $\color{blue}P_0$ and $\color{red}P_1$, indicated in Figure 3(c). Each is  a one-to-one continuous image of the half-line $[0,\infty)$ and is dense in $S$. Identifying their `endpoints' $\color{blue}\langle 0,0\rangle$ and $\color{red}\langle 1,1\rangle$ produces an indecomposable continuum $\dot S\coloneq S/\{\color{blue}\langle0,0\rangle\color{black},\color{red}\langle 1,1\rangle\color{black}\}$  which embeds into the plane $[0,1]^2\times \{\sfrac{1}{2}\}$ (Figure 4) and has composant  $\dot P\coloneq (\color{blue}P_0\color{black}\cup \color{red}P_1\color{black})/\{\color{blue}\langle0,0\rangle\color{black},\color{red}\langle 1,1\rangle\color{black}\}$.  

Let $f:(-\infty,\infty)\to \dot P$ be a one-to-one mapping onto $\dot P\subseteq [0,1]^2\times \{\sfrac{1}{2}\}$ such that $f\restriction [0,\infty)$ maps onto $\color{blue}P_0$ and $f\restriction (-\infty,0]$ maps onto $\color{red}P_1$. The point $f(0)$ joining the two original composants is indicated in Figure 4. 
  \begin{figure}[h]
        \centering
        \includegraphics[scale=.7]{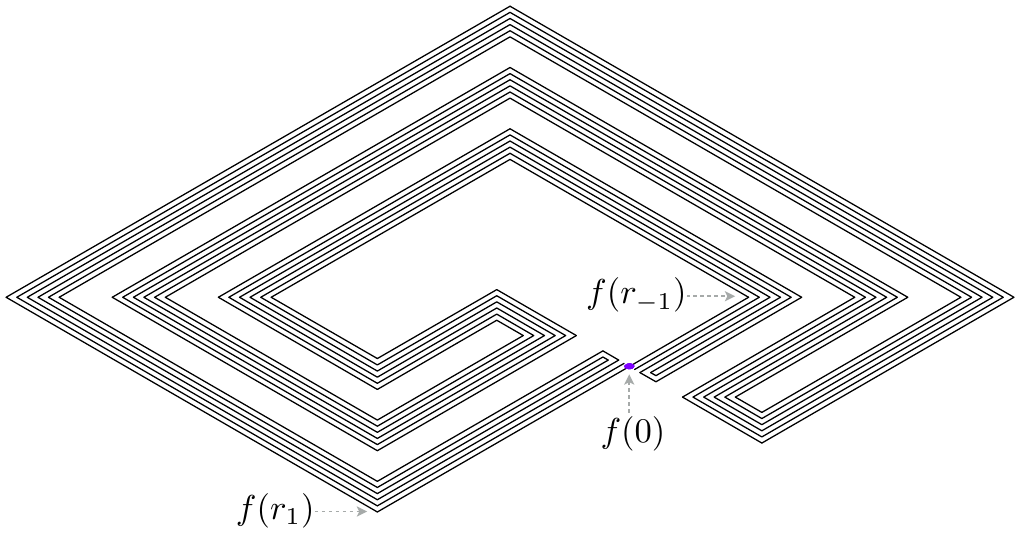} 
      \caption{$\dot S$}
        \label{fig:prob1_6_2}
 \vspace{-3mm}
\end{figure}
\begin{figure}[h]
    \centering
       \includegraphics[scale=0.47]{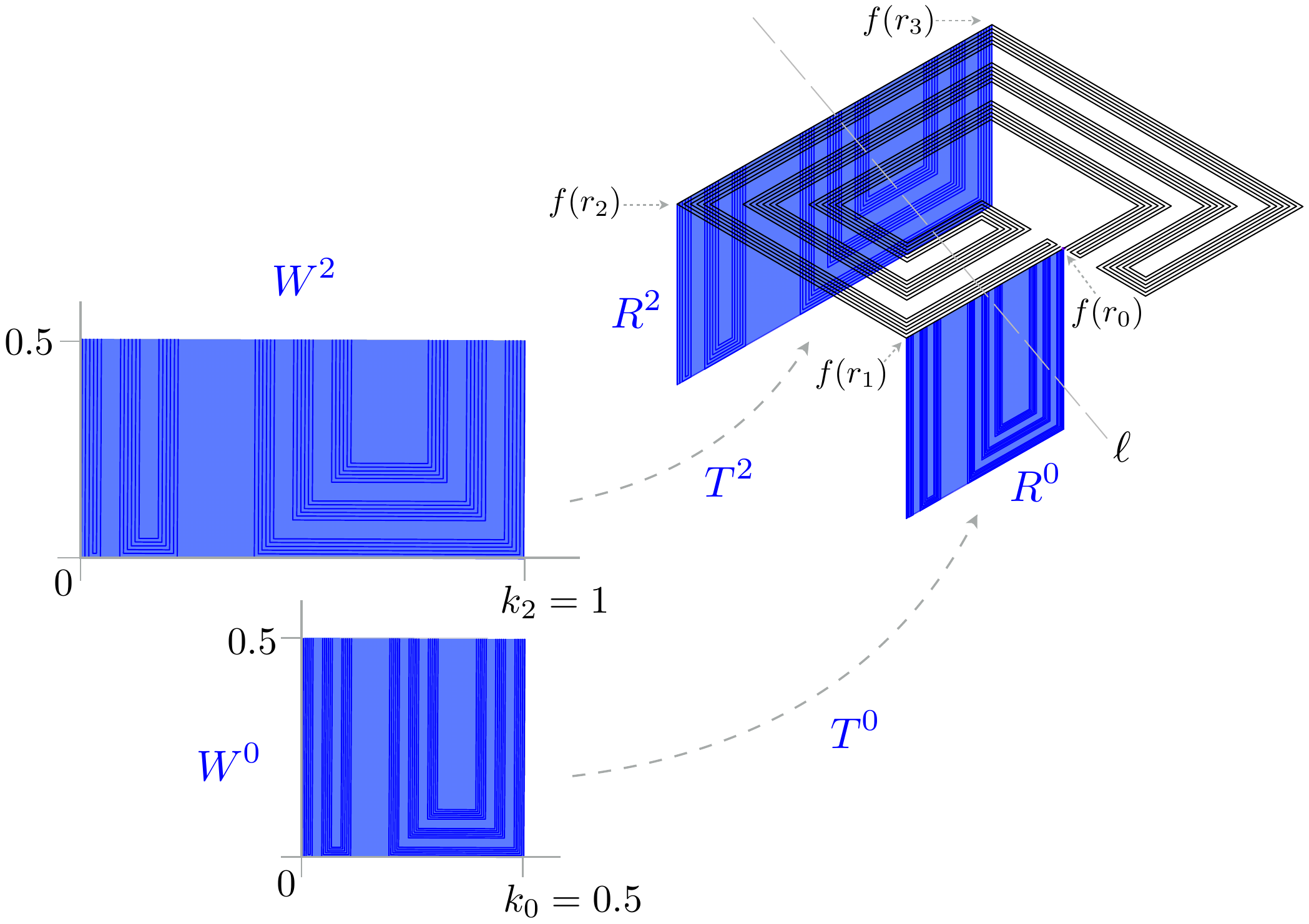}
   \caption{Attachment of parallel copies of $\color{blue}W_0$}
\end{figure}
 Let $(r_n)\in (-\infty,\infty)^\mathbb Z$ be the increasing $\mathbb Z$-enumeration of numbers in $(-\infty,\infty)$ such that $r_0=0$  and $f[r_{n-1},r_{n}]\perp f[r_{n},r_{n+1}]$ for $n\neq 0$. In  words, the points $f(r_n)$, $n\neq 0$, are the ninety-degree `corners' of $\dot P$.  


For each $n\geq 0$, let 
\begin{align*}
k_n&=\|f(r_{n})-f(r_{n+1})\|,\\
\color{blue}W^n\color{black}&=\{ \langle k_n \cdot x,y\rangle:\langle x,y\rangle\in \color{blue}W_0\color{black}\}, \text{and}\\
\color{blue}R^n\color{black}&=\{\langle x,y,z\rangle\in [0,1]^3:\langle x,y,\sfrac{1}{2}\rangle\in f[r_{n},r_{n+1}] \text{ and }z\in [0,\sfrac{1}{2}] \}.\end{align*}
Then  $\color{blue}W^n$ is a copy of $\color{blue}W_0$ horizontally scaled  by  factor  $k_n$ (the length of $f[r_n,r_{n+1}]$), and $\color{blue}R^n$ is a rectangular region extending below $f[r_n,r_{n+1}]$. For a given $n\geq 0$ there are two rigid transformations   $$\color{blue}T^n\color{black}: \color{black}[0,k_n]\times [0,\sfrac{1}{2}]\to\color{blue}R^n$$that map $[0,k_n]\times \{\sfrac{1}{2}\}$ onto $f[r_{n},r_{n+1}]$. One  transformation maps $\langle 0,\sfrac{1}{2}\rangle$  and $\langle k_n,\sfrac{1}{2}\rangle$ to $f(r_n)$ and $f(r_{n+1})$, respectively, while the other maps these points in the opposite way. By inductively choosing the `correct' $\color{blue}T^n\color{black}$, we can ensure  that whenever  $f[r_m,r_{m+1}]$  and $f[r_n,r_{n+1}]$ and are parallel ($0\leq m<n$), $\color{blue}T^n\color{black}[\color{blue}W^n\color{black}]$ and $\color{blue}T^m\color{black}[\color{blue}W^m\color{black}]$ are similarly aligned.  Specifically,   $\color{blue}T^n\color{black}\color{black}(\langle 0,\sfrac{1}{2}\rangle)$ and $\color{blue}T^m\color{black}\color{black}(\langle 0,\sfrac{1}{2}\rangle)$ always lie on the same side of the line $\ell\subseteq [0,1]^2\times \{\sfrac{1}{2}\}$ that bisects   the segments $f[r_m,r_{m+1}]$  and $f[r_n,r_{n+1}]$, and the points $\color{blue}T^n\color{black}(\langle k_n,\sfrac{1}{2}\rangle)$  and  $\color{blue}T^m\color{black}\color{black}(\langle k_m,\sfrac{1}{2}\rangle)$ lie on the other side of $\ell$. See Figure 5.

Likewise,  use rigid transformations  $\color{red}T^n\color{black}$, $n<0$, to attach horizontally scaled copies of $\color{red}W_1$  to $\color{red}P_1$.  Construct each $\color{red}T^n\color{black}$ so that $\color{red}T^n\color{black}[\color{red}W^n\color{black}\cap (\color{violet}[0,1]\times \{\sfrac{1}{2}\}\color{black})]$ is the middle two-fifths Cantor set  in  $f[r_{n},r_{n+1}]$, and $\color{red}T^n\color{black}[\color{red}W^n\color{black}]$ spans the rectangular region  $\color{red}R^n\color{black}:\approx f[r_{n},r_{n+1}]\times [\sfrac{1}{2},1]$. In the beginning,  $\color{red}T^{-1}\color{black}[\color{red}W^{-1}\color{black}]$ and $\color{red}T^{-2}\color{black}[\color{red}W^{-2}\color{black}]$ should be inversely aligned with respect to $\color{blue}T^0\color{black}[\color{blue}W^0\color{black}]$ and $\color{blue}T^1\color{black}[\color{blue}W^1\color{black}]$ (just as  $\color{red}W_1$ and $\color{blue}W_0$  are  aligned in $W$). Then,  choose the other $\color{red}T^n\color{black}$ ($n\leq-2$) so that the orientations of parallel $\color{red}T^n\color{black}[\color{red}W^n\color{black}]$ and $\color{red}T^m\color{black}[\color{red}W^m\color{black}]$ are the same.  This forces all parallel $\color{blue}T^n\color{black}[\color{blue}W^n\color{black}]$ and $\color{red}T^m\color{black}[\color{red}W^m\color{black}]$ ($n\geq 0$ and $m<0$) to be aligned as in $W$.  See Figure 7.  

Put $$\widehat W=\bigcup \{T^n[W^n]:n\in \mathbb Z\}.$$
\vspace{-7mm}
\begin{figure}[h]
    \begin{minipage}{0.5\textwidth}
        \centering
    \includegraphics[scale=0.45]{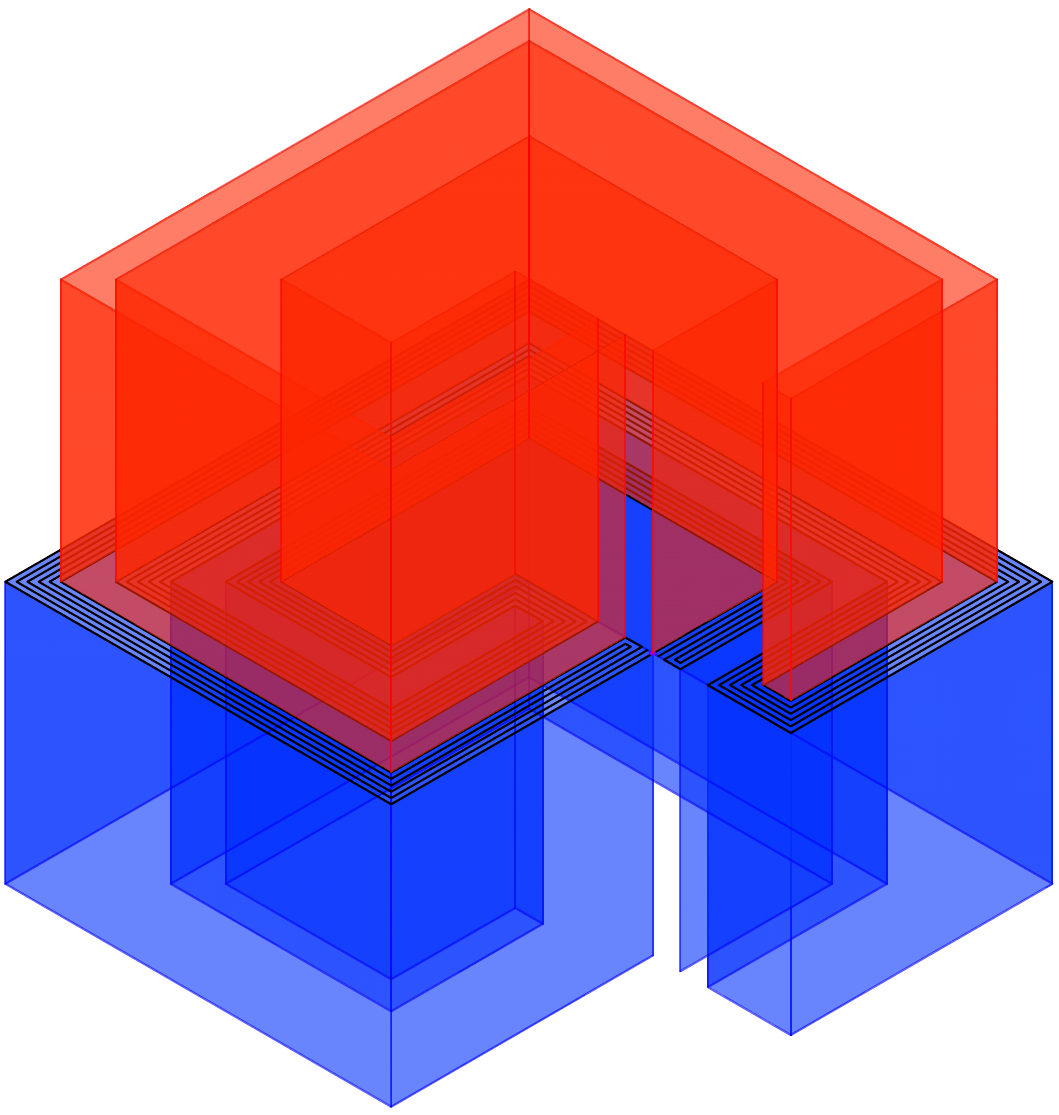} \;
      \caption[]{$\bigcup\{\dot S\cup R^n: -17\leq n\leq 16\}$}
        \label{fig:prob1_6_1}
    \end{minipage} 
     \begin{minipage}{0.49\textwidth}
     \vspace{5.8mm}
    \centering
       \includegraphics[scale=0.42]{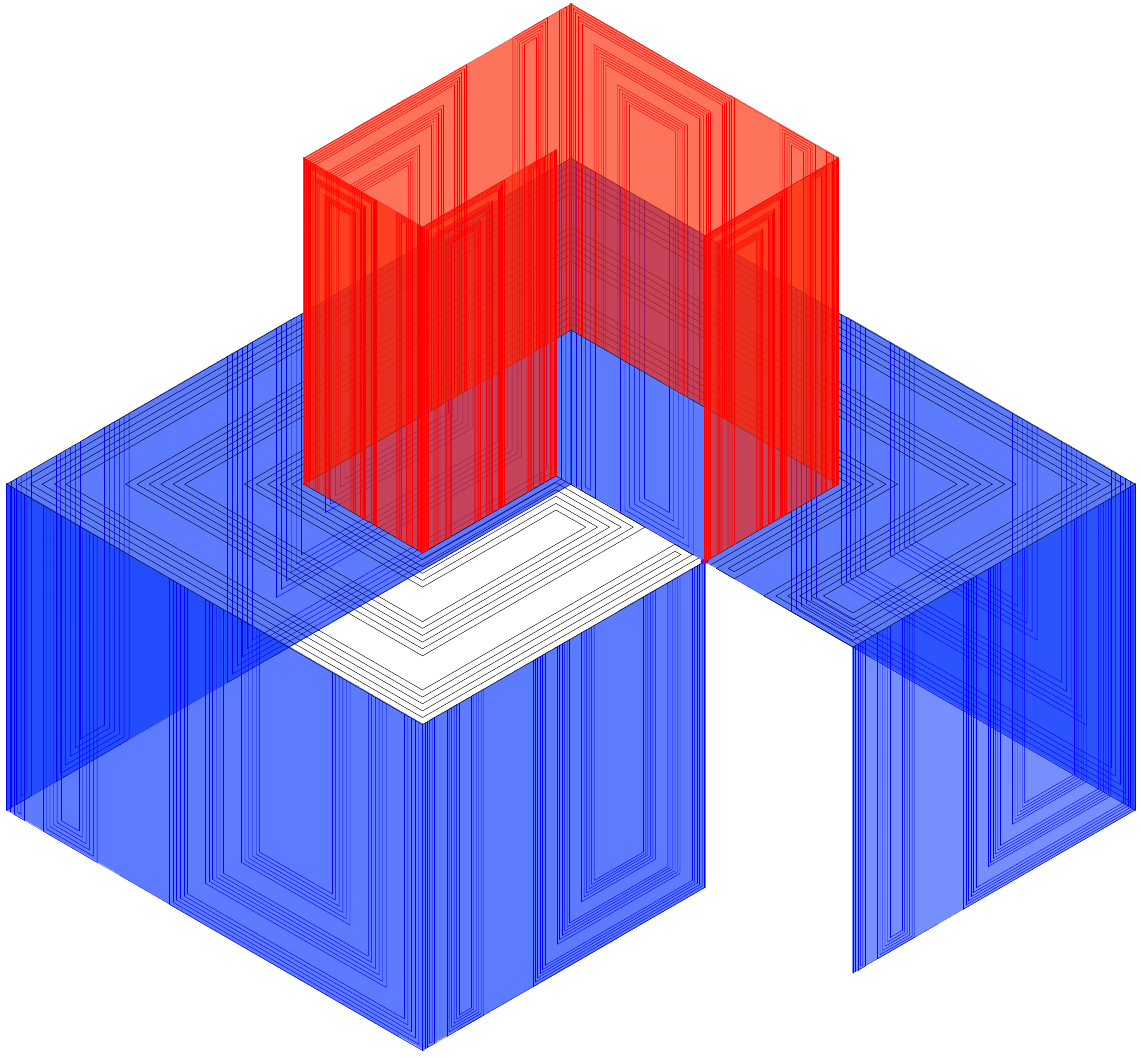}
   \caption{$\bigcup\{\dot S\cup T^n[W^n]:-5\leq n\leq 4\}$}
   \end{minipage}
\end{figure}

\subsection{Properties of $\widehat W$}

Verification of $\widehat W$'s properties will be kept fairly brief. The reader will be referred to \cite[Section 4, Example 4]{lip2} for supporting claims and proofs.

\begin{enumerate}
\item[i.] Since $W$ is connected and $\partial \color{blue}W_0\color{black}=\partial \color{red}W_1\color{black}=S\cap (\color{violet}[0,1]\times  \{\sfrac{1}{2}\}\color{black})$ is compact, Claims 1 and 2 in \cite[Example 4]{lip2} apply with $X'=W$ and $X_i'=W_i$.  They show the following.  

\noindent First, every clopen subset of the Cantor set $$\mathfrak C\coloneq (\{0\}\cup\{\sfrac{1}{n}:n=1,2,3,...\})\times [S\cap (\color{violet}[0,1]\times  \{\sfrac{1}{2}\}\color{black})]$$is eventually a product of clopen sets.  That is, if $C$ is a clopen subset of  $\mathfrak C$  then there is a sufficiently large positive integer $m$ such that $$C\cap ([0,\sfrac{1}{m}]\times S\cap (\color{violet}[0,1]\times  \{\sfrac{1}{2}\}\color{black}))=[0,\sfrac{1}{m}]\times \{s\in S\cap (\color{violet}[0,1]\times  \{\sfrac{1}{2}\}\color{black}):\langle \sfrac{1}{m},s\rangle\in C\}.$$Next, suppose $i\in \{0,1\}$ and let $\mathfrak W\supseteq \mathfrak C$ be the set $$(\{0\}\times W_i)\cup (\{\sfrac{1}{n}:n=1,2,3,...\}\times W_{1-i}).$$ If $A$ is a clopen subset  of $\mathfrak W$ such that $A\cap (\{0\}\times W_i)\neq\varnothing$, then,  letting $m$ be as above (for $C=A\cap \mathfrak C$), we find that $[A\cap (\{0\}\times W_i)]\cup [A\cap (\{\sfrac{1}{m}\}\times W_{1-i})]$ projects onto a clopen subset of $\{0\}\times W$.  Since $W$ is connected, $(\{0\}\times W_i)\subseteq A$.  This argument shows that $\{0\}\times W_i$ is contained in a quasi-component of $\mathfrak W$.  

\noindent Back in $\widehat W$,  density of the $P_i$ and  alignment of the $\color{blue}T^n\color{black}[\color{blue}W^n\color{black}]$ and $\color{red}T^m\color{black}[\color{red}W^m\color{black}]$  now imply that  each $T^n[W^n]$ is contained in a quasi-component of $\widehat W$ (each half of $W$ in $\widehat W$ is the ``limit'' of a sequence of complementary halves which are properly oriented  as in $\mathfrak W$). It follows   that $\widehat W$ has only one quasi-component, i.e. $\widehat W$ is connected (Claim 3 of \cite[Example 4]{lip2}). 

\item[ii.] The closure of $\widehat W$ is an indecomposable continuum by the facts that  $\dot S$ is indecomposable and  $S$ is irreducible between $\{0\}\times [0,1]$ and $\{1\}\times [0,1]$  (see the proof of Claim 4 in \cite[Example 4]{lip2}).  

Moreover, if $L$ is a proper subcontinuum of $\cl_{[0,1] ^3}\widehat W$, then its orthogonal projection into $\dot P$ is contained in an arc.  This means for every proper component  $X\subseteq \widehat W$ there exists $m\in \mathbb N$ such that $X\subseteq \bigcup \{T^n[W^n]:|n|\leq m\}$.  Each $T^n[W^n]$ is hereditarily disconnected, as are the joining sets $$(\{f(r_n)\}\times [0,1])\cap (T^{n-1}[W^{n-1}]\cup T^n[W^{n}]).\footnote{For certain $W$, it is possible that $(\{f(r_n)\}\times [0,1])\cap (T^{n-1}[W^{n-1}]\cup T^n[W^{n}])$ contains an arc. But with the `Knaster-Kuratowski' type $W$ we have $\widehat W\cap (\{f(r_n)\}\times [0,1])\simeq \mathbb Q$.} $$ Thus $\bigcup \{T^n[W^n]:|n|\leq m\}$ decomposes into a Cantor set of hereditarily disconnected fibers.   
It follows that $X$ is degenerate, establishing the widely-connected property of $\widehat W$.

\item[iii.] Every continuum containing $\widehat W$ is reducible between every two points of $\widehat W$. Indeed, let  $Y\supseteq \widehat W$ be any continuum, and let $p,q\in \widehat W$.  There exists $m<\omega$ such that $\{p,q\}\subseteq  \bigcup \{T^n[W^n]:|n|\leq m\}$. By the reasoning in part i, as well as the fact that the quasi-components of closed subsets of $Y$ are connected, $\bigcup \{T^n[W^n]:|n|\leq m\}$ is contained in a proper subcontinuum of $Y$.  So $Y$ is reducible between $p$ and $q$. For further details, see  the proof of Claim 5 in \cite[Example 4]{lip2}.   
\end{enumerate}

\subsection{Plane embedding}

Notice that a slight alteration to $\widehat W$ makes its closure  a chainable  continuum. Indeed, let $P$ be any composant of $S$ other than $P_0$ and $P_1$.  Then $P$ is a one-to-one continuous image of $(-\infty,\infty)$, and may be used instead of $\dot P$ to construct widely-connected set $\wideparen W \approx \widehat W $.  To see that $\text{cl}_{\mathbb R ^3} \wideparen{W}$ is chainable,  use chainability of the center template $S\times \{\sfrac{1}{2}\}$ together with chainablility of the vertical copies of $S$ in $\text{cl}_{\mathbb R ^3}\wideparen W$.  Every chainable continuum is planar, so   $\wideparen W$ embeds into the plane.


The continuum $\cl_{[0,1]^3}\widehat W$ is not chainable, but it is circle-like. Such a continuum does not automatically embed into the plane, but Bing \cite[Theorem 4]{bin} proved: \textit{If $X$ is a circle-like continuum and there is a sequence $\mathscr C_0, \mathscr C_1, . . .$  of circular chains covering $X$ such that  $\text{mesh}(\mathscr C_n)\to 0$ as $n\to\infty$, and $\mathscr C_{n+1}$ circles $\mathscr C_n$ exactly once, then $X$ homeomorphically embeds into the plane.} According to the definitions in Bing's paper, $X=\dot S$ is circularly chainable in the appropriate manner. Using chainability of the vertical copies of $S$ in $\cl_{[0,1]^3}\widehat W$, one can see that $\cl_{[0,1]^3}\widehat W$ also satisfies the hypothesis of Bing's theorem.

\section{$\leq 2$-point compactifications}

The interval $[0,1)$ is the only locally connected \& locally compact connected space whose  one-point compactification is irreducible.  For suppose $X$ and its one-point compactification $\alpha X\coloneq X\cup \{\infty\}$ have the stated properties. By   \cite[Theorem 4.1]{groot}, $\alpha X$ is locally connected.   Every locally connected irreducible continuum is homeomorphic to $[0,1]$  (cf.  \cite{kur}), so $\alpha X\simeq [0,1]$. Since $X$ is connected, $\infty$ corresponds to an endpoint of $[0,1]$.  Thus $X\simeq [0,1)$.   A similar  line of reasoning will show that $\mathbb R$ is the only locally connected \& locally compact connected set with an irreducible two-point compactification.

\section{Questions}

\begin{uq}Let $X$ be indecomposable and reducible.  If $p$ is a limit point of $X$ which is not in the closure of any proper component of $X$, then is $X\cup \{p\}$ necessarily (a)  indecomposable? (b) irreducible?\end{uq}

\begin{uq}Let $f:[0,\infty)\to \mathbb R ^2$ be a one-to-one mapping of the half-line into the plane, with range $X\coloneq f[0,\infty)$. If $\overline{f[n,\infty)}=\overline X$ for each $n<\omega$, then is $X\cup \{p\}$ irreducible for every  point $p\in \overline X\setminus X$?   \end{uq}

\end{document}